\documentclass[]{amsart}
\usepackage{amssymb}
\usepackage{amsthm}
\usepackage{amsmath}
\usepackage{mathtools}
\usepackage{bbm}
\usepackage{enumitem}
\usepackage{lmodern}
\usepackage[usenames,dvipsnames]{color}
\usepackage{hyperref}
\hypersetup{
 colorlinks,
 linkcolor={red!50!black},
 citecolor={blue!50!black},
 urlcolor={blue!80!black}
}
\usepackage[a4paper, inner=3cm,outer=3cm,top=3.3cm,bottom=3cm]{geometry}
\usepackage{subcaption}
\usepackage{graphicx}
\usepackage{pgfplots}
\usepackage{tikz,tikz-3dplot}
\usetikzlibrary{cd}
\usetikzlibrary{shapes.misc}
\usetikzlibrary{arrows,decorations.markings}
\usepackage{xifthen}

\AtBeginDocument{%
\def\MR#1{}
}
\theoremstyle{definition}
\newtheorem{thm}{Theorem}
\newtheorem{prop}[thm]{Proposition}
\newtheorem{cor}[thm]{Corollary}
\newtheorem{lemma}[thm]{Lemma}

\newtheorem{question}[thm]{Question}

\newtheorem{ex}[thm]{Example}

\newcommand{\Q}{{\mathbb{Q}}}

\newcommand{\PP}{{\mathbb{P}}}

\newcommand{\rvline}{\hspace*{-\arraycolsep}\vline\hspace*{-\arraycolsep}}

\DeclareMathOperator{\Sym}{Sym}

\DeclareMathOperator{\chr}{char}

\DeclareMathOperator{\tr}{tr}
\DeclareMathOperator{\pdim}{pdim}
\DeclareMathOperator{\init}{in}
\DeclareMathOperator{\HF}{HF}
\DeclareMathOperator{\reg}{reg}
\DeclareMathOperator{\height}{ht}

\DeclareMathOperator{\Proj}{Proj}
\title{Ideals of submaximal minors of sparse symmetric matrices}

\author[J.\,Deng]{Jiahe Deng}

\author[A.\,Kretschmer]{Andreas Kretschmer}
\address[A.\,Kretschmer]{Institut f\"ur Algebra und Geometrie\\Fakult\"at f\"ur Mathematik\\Otto-von-Guericke-Universit\"at Magdeburg}
\curraddr{}
\email{andreas.kretschmer@ovgu.de}
\thanks{}

\begin{document}
\maketitle
\setlength{\parindent}{0pt}
\nocite{*}

\begin{abstract}
	We study algebraic and homological properties of the ideal of submaximal minors of a sparse generic symmetric matrix. This ideal is generated by all $(n-1)$-minors of a symmetric $n \times n$ matrix whose entries in the upper triangle are distinct variables or zeros and the zeros are only allowed at off-diagonal places. The surviving off-diagonal entries are encoded as a simple graph $G$ with $n$ vertices. We prove that the minimal free resolution of this ideal is obtained from the case without any zeros via a simple pruning procedure, extending methods of Boocher. This allows us to compute all graded Betti numbers in terms of $n$ and a single invariant of $G$. Moreover, it turns out that these ideals are always radical and have Cohen--Macaulay quotients if and only if $G$ is either connected or has no edges at all. The key input are some new Gröbner basis results with respect to non-diagonal term orders associated to $G$.
\end{abstract}

\section{Introduction and results}

The ideal of maximal minors of a generic $m \times n$ matrix is one of the most studied objects in combinatorial commutative algebra due to its geometric significance and its rich combinatorial structure. If one allows the matrix to be \emph{sparse}, i.e.~to have some entries replaced by zero, some important tools of study disappear, most notably since the action of the general linear group does of course not preserve the zero pattern. Nonetheless, many results have been obtained in the sparse case, for example by Giusti and Merle \cite{GiustiMerle}, Boocher \cite{Boocher2012Paper} and Conca and Welker \cite{Conca2019Lovasz}, and in part even for much weaker assumptions on the entries of the matrix (see for example \cite{Conca2015Universal,Conca2018Cartwright,Conca2020Universal2,Conca2022Radical,Eisenbud1988Linear,Miro-Roig2007Betti,Miro-Roig2008Book}). The symmetric case, i.e.~the ideal of minors of fixed size of a sparse generic \emph{symmetric} matrix, is much less understood. In \cite{Conca2019Lovasz}, Conca and Welker study geometric and arithmetic properties like primality, reducedness, codimension and the complete intersection property of these ideals and of other kinds of sparse determinantal ideals. Homological invariants, on the other hand, such as the graded Betti numbers, projective dimension and regularity are in general still unknown in the sparse symmetric case (see, however, the related \cite[Section~3]{Miro-Roig2007Betti}). Our main result fills this gap in the case of submaximal minors of a sparse generic symmetric matrix, using ideas of Boocher \cite{Boocher2012Paper, Boocher2013Thesis} together with a foundational result by Józefiak \cite{Jozefiak1978Ideals} and combining these with new Gröbner basis results.

More precisely, let $K$ be a field and $R = K[x_{ij}: 1 \leq i \leq j \leq n]$ the polynomial ring, $n \geq 2$. Let $X = (x_{ij})$ be the generic symmetric $n \times n$ matrix, i.e. $x_{ij} \coloneqq x_{ji}$ for $i > j$. Let $G$ be an undirected simple graph with vertex set $[n] \coloneqq \{1,2,\ldots,n\}$ and let $Z$ be the set of all off-diagonal variables corresponding to the non-edges of $G$. We define $X_G$ to be the matrix obtained from $X$ by substituting zeros for all variables in $Z$, and write $I_{n-1}(X)$ and $I_{n-1}(X_G)$ for the ideals of $R$ generated by all $(n-1)$-minors of $X$ and $X_G$, respectively.

If $C_1, \ldots, C_r$ is the partition of $[n]$ where the $C_i$ are the vertex sets of the connected components of $G$, we define
\begin{equation*}
    D_G \coloneqq \sum_{1 \leq s < t \leq r} |C_s| \cdot |C_t|.
\end{equation*}
Then $0 \leq D_G \leq \binom{n}{2}$, and the lower bound is attained if and only if $G$ is connected while the upper bound is attained if and only if $G$ has no edges at all. The following is our main result.

\vbox{
\begin{thm}\label{thm:main}
    The minimal graded free resolution of $R/I_{n-1}(X_G)$ is obtained from the one of $R/I_{n-1}(X)$ by substituting zeros for all variables in $Z$ and ``pruning'' the resulting complex. The graded Betti numbers are
    \begin{align*}
        \beta_{1,n-1}(R/I_{n-1}(X_G)) &= \binom{n+1}{2} - D_G, \\
        \beta_{2,n}(R/I_{n-1}(X_G)) &= n^2 - 1 - 2D_G, \\
        \beta_{3,n+1}(R/I_{n-1}(X_G)) &= \binom{n}{2} - D_G.
    \end{align*}
    All other graded Betti numbers (apart from $\beta_{0,0} = 1$) are zero. Therefore, $I_{n-1}(X_G)$ has a linear resolution with $\reg(I_{n-1}(X_G)) = n-1$. We also deduce $\pdim(R/I_{n-1}(X_G)) = 3$ except if $G$ has no edges at all, in which case $\pdim(R/I_{n-1}(X_G)) = 2$. The quotient ring $R/I_{n-1}(X_G)$ is reduced, and it is Cohen--Macaulay if and only if $G$ is either connected or has no edges at all. Finally,
    \begin{equation*}
        \height(I_{n-1}(X_G)) = \begin{cases} 3 & \text{if } G \text{ is connected}, \\ 2 & \text{otherwise}, \end{cases}
    \end{equation*}
    so that $I_{n-1}(X_G)$ is a perfect ideal if and only if $G$ is either connected or has no edges at all.
\end{thm}
}

The proof of Theorem~\ref{thm:main} is contained in Sections~\ref{section_groebner} and~\ref{section_minimal_free_res}. The pruning procedure we refered to is the same as the one used by Boocher in \cite{Boocher2012Paper, Boocher2013Thesis}: First, we set to zero all variables in $Z$ in all three matrices of the minimal free resolution of $R/I_{n-1}(X)$. Next, all zero columns of the first matrix of the resolution together with the corresponding rows of the second matrix are erased. Then, all emerging zero columns of the cropped second matrix together with the corresponding rows of the third matrix are erased. Finally, we also delete all emerging zero columns of the cropped third matrix. Theorem~\ref{thm:main} implies that the resulting complex is again exact and that, if $G$ is connected, no pruning occurs at all.

If one is only interested in the case where $G$ is connected, then the arguments of Section~\ref{section_minimal_free_res} can be avoided. This is because it follows from Section~\ref{section_groebner} that $I_{n-1}(X_G)$ has grade $3$ (since grade and height agree for ideals in regular rings). Then J\'ozefiak's result \cite[Theorem~3.1]{Jozefiak1978Ideals} yields that the minimal free resolution of $I_{n-1}(X)$ stays exact after substituting zeros for all variables in~$Z$.

\section{Gröbner bases}\label{section_groebner}

We keep the notation of the introduction. Given a graph $G$ on the vertex set $[n]$, let $w_G$ be the weight vector for the polynomial ring $R = K[x_{ij}: 1 \leq i \leq j \leq n]$ which assigns the following weights to the variables:
\begin{alignat*}{2}
    w_G(x_{ii}) &= 2 \quad &&\text{for all } i = 1, \ldots, n, \\
    w_G(x_{ij}) &= 2 &&\text{for all } ij \in G, \\
    w_G(x_{ij}) &= 1 &&\text{for all } ij \not\in G.
\end{alignat*}
In other words, the variables in $Z$ have $w_G$-weight $1$ while all others have $w_G$-weight $2$. This is analogous to the weights used by Boocher \cite{Boocher2012Paper, Boocher2013Thesis}. Moreover, by $w_{\text{diag}}$ we denote the weight vector which assigns weight $2$ to the diagonal variables and weight $1$ to the off-diagonal variables (which is the special case of $w_G$ where $G$ has no edges). Let $T \subseteq G$ be a spanning forest, i.e., the union of one spanning tree for each connected component of $G$. The main goal of this section is to show that the $\binom{n+1}{2}$ standard $(n-1)$-minors generating $I_{n-1}(X)$ form a Gröbner basis with respect to the weight order given by $<_{T,G} \ \coloneqq \ <_{w_\text{diag}} \circ <_{w_T} \circ <_{w_G}$ and that $\init_{<_{T,G}}(I_{n-1}(X))$ is a square-free monomial ideal. Here, the composition of the weight orders is to be read from right to left, i.e., two monomials are first compared with respect to their $w_G$-degrees, then, in case of equal $w_G$-degrees, with respect to their $w_T$-degrees, and finally, in case of equal $w_T$-degrees, with respect to $w_\text{diag}$. Of course, this is not a monomial order since it does not yield a total ordering of the set of all monomials in $R$. Nonetheless, the ideals of leading terms of $I_{n-1}(X)$ and $I_{n-1}(X_G)$ with respect to $<_{T,G}$ are indeed monomial, as we shall see.

We define $I_T \subseteq R$ to be the square-free monomial ideal generated by the following monomials:
\begin{align*}
\frac{x_{11} \cdots x_{nn}}{x_{ii}} \quad &\text{for all } i = 1, \ldots, n, \\
\frac{x_{11} \cdots x_{nn}}{x_{kk} x_{ll}} x_{kl} \quad &\text{for } 1 \leq k < l \leq n \text{ if in } T \text{ there is no path between } k \text{ and } l, \\
\left(\prod_{i \in [n] \setminus V(p)} x_{ii}\right) \cdot x_p \quad &\text{for all paths } p \text{ in } T.
\end{align*}
We use the word path in its strictest sense, i.e., no vertex appears twice in a path and it contains at least one edge. By $V(p)$ we denote the vertex set of the path $p$ and by $x_p \coloneqq \prod_{ij \in p} x_{ij}$ the product of all off-diagonal variables corresponding to the edges in $p$. It is easily seen that all generators are square-free monomials of degree $n-1$, none is redundant in $I_T$ and their total number is $\binom{n+1}{2}$. The latter follows from realizing that in a tree every path is uniquely determined by its two end points. We have $I_T \subseteq \init_{<_{T,G}}(I_{n-1}(X))$ since the generators of $I_T$ are precisely the initial terms of the standard $(n-1)$-minors with respect to $<_{T,G}$. This can easily be read off the following formula.

\begin{prop}[\cite{jones2005covariance,DoubleMarkovian}]\label{prop:pathDet}
For all $1 \leq k < l \leq n$ we have
  \begin{equation*}
    (-1)^{k+l}\det \left( (X_G)_{[n] \setminus k, [n] \setminus l} \right) = \sum_{\substack{p \text{ path in } G \\ \text{ between } k \text{ and } l}} (-1)^{|V(p)|-1} \cdot \det \left((X_G)_{[n] \setminus V(p), [n] \setminus V(p)} \right) \cdot x_p.
  \end{equation*}
\end{prop}

\begin{proof}
    This is the content of both~\cite[Theorem~1]{jones2005covariance} and~\cite[Proposition~3.19]{DoubleMarkovian}.
\end{proof}

In the special case where $G = T$ has no edges at all, we simply write $I$ for $I_T$,
\begin{equation*}
I = \left(\frac{x_{11} \cdots x_{nn}}{x_{ii}}, \frac{x_{11} \cdots x_{nn}}{x_{kk} x_{ll}} x_{kl} \ | \ i = 1, \ldots, n, \ 1 \leq k < l \leq n \right) \subseteq R.
\end{equation*}
In this case, $<_{T,G} \ = \ <_{w_\text{diag}}$.

\begin{lemma}
	We have $I = \init_{w_\text{diag}}(I_{n-1}(X))$, in particular $\HF(I) = \HF(I_{n-1}(X))$ and the standard $(n-1)$-minors of $X$ form a Gröbner basis of $I_{n-1}(X)$ with respect to $<_{w_\text{diag}}$.
\end{lemma}

\begin{proof}
	Since the generators of $I$ are the initial terms of the standard minors generating $I_{n-1}(X)$ with respect to $<_{w_\text{diag}}$, we have $I \subseteq \init(I_{n-1}(X))$. Hence, it is enough to prove $\HF(I) = \HF(I_{n-1}(X))$. The Hilbert function of $I_{n-1}(X)$ may be deduced from the minimal graded free resolution
    \begin{equation*}
        0 \rightarrow R(-(n+1))^{\binom{n}{2}} \rightarrow R(-n)^{n^2 - 1} \rightarrow R(-(n-1))^{\binom{n+1}{2}} \rightarrow I_{n-1}(X) \rightarrow 0,
    \end{equation*}
    provided by J\'ozefiak \cite[Theorem~3.1]{Jozefiak1978Ideals}, see Section~\ref{section_minimal_free_res} for more details. One obtains
    \begin{align*}
        \HF(I_{n-1}(X))(d) &= \binom{n+1}{2} \HF(R)(d-n+1) - (n^2 - 1) \HF(R)(d-n) + \binom{n}{2} \HF(R)(d-n-1) \\
        &= \binom{n+1}{2} \binom{\binom{n}{2} + d}{\binom{n+1}{2} - 1} - (n^2 - 1) \binom{\binom{n}{2}-1+d}{\binom{n+1}{2} - 1} + \binom{n}{2} \binom{\binom{n}{2}-2+d}{\binom{n+1}{2} - 1} \\
        &= \binom{\binom{n}{2} + d}{\binom{n+1}{2} - 1} + (n-1) \binom{\binom{n}{2} - 1 + d}{\binom{n+1}{2} - 2} + \binom{n}{2} \binom{\binom{n}{2} - 2 + d}{\binom{n+1}{2} - 3},
    \end{align*}
	where in the last step we used the binomial identity $\binom{m}{k} = \binom{m-1}{k} + \binom{m-1}{k-1}$ several times. On the other hand, we compute $\HF(R/I)$ via the technique described in \cite[Section~15.1.1]{Eisenbud2013Commutative}. We first choose some off-diagonal generator of $I$, say $\frac{x_{11} \cdots x_{nn}}{x_{kk} x_{ll}} x_{kl}$ with $k < l$. Consider the associated short exact sequence
    \begin{equation*}
        0 \rightarrow (R/(x_{kk}, x_{ll}))(-(n-1)) \rightarrow R/I_1 \rightarrow R/I \rightarrow 0,
    \end{equation*}
    where the first map is multiplication by the chosen generator and $I_1$ has the same generators as $I$ except for the latter. We obtain $\HF(R/I)(d) = \HF(R/I_1)(d) - \HF(R/(x_{kk},x_{ll}))(d-n+1)$. Since $R/(x_{kk},x_{ll})$ is a polynomial ring in $\binom{n+1}{2}-2$ variables, its Hilbert function is $p(d) \coloneqq \binom{\binom{n+1}{2}-3+d}{\binom{n+1}{2}-3}$, and we can proceed with $R/I_1$. In the same way, we pick any remaining off-diagonal generator of $I_1$, and from the analogous short exact sequence we get $\HF(R/I_1)(d) = \HF(R/I_2)(d) - p(d-n+1)$, i.e., $\HF(R/I)(d) = \HF(R/I_2)(d) - 2p(d-n+1)$ etc. We continue in this manner until all off-diagonal generators are gone. We then have $\HF(R/I)(d) = \HF(R/J)(d) - \binom{n}{2} p(d-n+1)$, where $J = \left( \frac{x_{11} \cdots x_{nn}}{x_{ii}} \ | \ i = 1, \ldots, n \right) \subseteq R$. Picking the first generator of $J$, i.e. $x_{22} \cdots x_{nn}$, we get the short exact sequence
    \begin{equation*}
        0 \rightarrow (R/(x_{11}))(-(n-1)) \rightarrow R/J_1 \rightarrow R/J \rightarrow 0,
    \end{equation*}
    where $J_1$ has the same generators as $J$ except for the first one. This shows $\HF(R/J)(d) = \HF(R/J_1)(d) - \HF(R/(x_{11}))(d-n+1)$. The Hilbert function of the polynomial ring $R/(x_{11})$ is again known, so we may proceed with $R/J_1$. The analogous short exact sequence occurs $n-1$ times until we reach $J_{n-1}$ which is principal. The last short exact sequence then has the form
    \begin{equation*}
        0 \rightarrow R(-(n-1)) \rightarrow R \rightarrow R/J_{n-1} \rightarrow 0.
    \end{equation*}
    We obtain
    \begin{equation*}
        \HF(R/J)(d) = \binom{\binom{n+1}{2}-1+d}{\binom{n+1}{2}-1} - (n-1) \binom{\binom{n}{2}-1+d}{\binom{n+1}{2}-2} - \binom{\binom{n}{2}+d}{\binom{n+1}{2}-1}.
    \end{equation*}
    Therefore, for the Hilbert function of $I$ we have
    \begin{align*}
        \HF(I)(d) &= \HF(R)(d) - \HF(R/I)(d) \\
        &= \HF(R)(d) - \HF(R/J)(d) + \binom{n}{2} \binom{\binom{n}{2}-2+d}{\binom{n+1}{2}-3} \\
        &= \HF(I_{n-1}(X))(d),
    \end{align*}
    as desired.
\end{proof}

The following lemma is the key technical input for the proof of Theorem~\ref{thm:main}.

\begin{lemma}\label{lemma:I_T}
	For any spanning forest $T$ on $[n]$ we have $\HF(I) = \HF(I_T)$.
\end{lemma}

\begin{proof}
	As both ideals are square-free monomial ideals in the same polynomial ring $R$, we can consider them as the defining ideals of two Stanley--Reisner rings for two simplicial complexes on the same vertex set. The Hilbert functions of these agree if and only if the face numbers of the simplicial complexes agree in every dimension, or equivalently if in every degree the numbers of \emph{square-free} monomials contained in $I$ resp. $I_T$ agree. Hence, for each $d$ we now construct a bijection $f_d \colon \mathcal{A}_{I,d} \rightarrow \mathcal{A}_{I_T,d}$ from the set $\mathcal{A}_{I,d}$ of square-free monomials of degree $d$ in $I$ to the set $\mathcal{A}_{I_T,d}$ of square-free monomials of degree $d$ in $I_T$. We first observe that every square-free monomial of $R$ which is divisible by at least $n-1$ distinct diagonal variables lies in both $I$ and $I_T$. Therefore, we define $f_d$ as the identity on the set of all such square-free monomials of degree $d$. Our second observation is that every square-free monomial in $I$ which is not divisible by $n-1$ distinct diagonal variables, is divisible by precisely $n-2$ distinct diagonal variables, and every such square-free monomial is a multiple of a \emph{unique} generator of $I$, say of $\frac{x_{11} \cdots x_{nn}}{x_{kk} x_{ll}} x_{kl}$. Let $m$ be a square-free monomial consisting only of off-diagonal variables different from $x_{kl}$. We now explain which monomial $f_d(m \cdot \frac{x_{11} \cdots x_{nn}}{x_{kk} x_{ll}} x_{kl})$ should be. If $k$ and $l$ lie in different connected components of $T$, then we simply let $f_d(m \cdot \frac{x_{11} \cdots x_{nn}}{x_{kk} x_{ll}} x_{kl}) = m \cdot \frac{x_{11} \cdots x_{nn}}{x_{kk} x_{ll}} x_{kl}$. Otherwise, let $p$ be the unique path in $T$ between $k$ and $l$. Write $p = k_0 k_1 \cdots k_r$ with $k_0 = k$, $k_r = l$. We write $m = m_1 \cdot m_2 \cdot m_3$. Here, $m_1$ is the product of all variables in $m$ which do not correspond to edges in $p$, i.e., $m_1$ does not involve any of  the variables $x_{k_0, k_1}, x_{k_1, k_2}, \ldots, x_{k_{r-1}, l}$. Next, $m_2$ is the product of all variables in $m$ corresponding to the edges of $p$ except for the last edge $k_{r-1} l$. Finally, $m_3$ is either $x_{k_{r-1}, l}$ or $1$, depending on whether $m$ is divisible by $x_{k_{r-1}, l}$ or not. We define $m_1' \coloneqq m_1$ and $m_2'$ to be the product of all diagonal variables $x_{ii}$ such that $i = k_j$ and $x_{k_{j-1}, k_j}$ divides $m_2$. Moreover, we set $m_3' = 1$ if $m_3 = 1$ and $m_3' = x_{kl}$ otherwise, i.e., if $m_3 = x_{k_{r-1}, l}$. Then, we put
    \begin{equation*}
        f_d \colon \quad \frac{x_{11} \cdots x_{nn}}{x_{kk} x_{ll}} x_{kl} \cdot m_1 \cdot m_2 \cdot m_3 \quad \mapsto \quad \left(\prod_{i \in [n] \setminus V(p)} x_{ii}\right) \cdot x_p \cdot m_1' \cdot m_2' \cdot m_3'.
    \end{equation*}
    Observe that the right hand side is indeed a square-free monomial in $I_T$ of the same degree as the left hand side, and the right hand side is divisible by at most $n-2$ distinct diagonal variables. Indeed, the right hand side is never divisible by $x_{kk}$ and $x_{ll}$, by construction. This observation even shows that $f_d$ is injective. Indeed, it suffices to show injectivity after restricting to the subset of $\mathcal{A}_{I,d}$ of those square-free monomials divisible by exactly $n-2$ distinct diagonal variables. Assume $f_d$ maps two distinct monomials $m \cdot \frac{x_{11} \cdots x_{nn}}{x_{kk} x_{ll}} x_{kl}$ and $n \cdot \frac{x_{11} \cdots x_{nn}}{x_{aa} x_{bb}} x_{ab}$ to the same element in $\mathcal{A}_{I_T,d}$. If $k$ and $l$ lie in different connected components of $T$ and the same is true for $a$ and $b$, then clearly this is not possible. If $k$ and $l$ lie in the same connected component, then it follows that $\left(\prod_{i \in [n] \setminus V(p)} x_{ii}\right) \cdot x_p \cdot m_1' \cdot m_2' \cdot m_3'$ is not divisible by $x_{aa}$ and $x_{bb}$ which implies that $a$ and $b$ are both vertices of the path $p$, so this is only possible if also $a$ and $b$ lie in the same connected component as each other which must also be the connected component of $k$ and $l$. But then, the unique path from $a$ to $b$ in $T$ is a subpath of the path $p$ from $k$ to $l$, and the converse is true as well, hence $ab = kl$. But for fixed $kl$, the assignment $m_1 \cdot m_2 \cdot m_3 \mapsto m_1' \cdot m_2' \cdot m_3'$ is injective by construction, proving the claim. As surjectivity of $f_d$ is clear, this concludes the proof.
\end{proof}

\begin{cor}\label{cor:generic_minors_groebner}
    The $(n-1)$-minors of $X$ form a Gröbner basis of $I_{n-1}(X)$ with respect to $<_{T,G}$. Equivalently, $\init_{<_{T,G}}(I_{n-1}(X)) = I_T$.
\end{cor}

\begin{lemma}
    Let $J \subseteq K[x_1, \ldots, x_m]$ be an ideal and $g_1, \ldots, g_r$ a Gröbner basis of $J$ with respect to a monomial order $<$. Let $Z$ be a subset of the variables $x_1, \ldots, x_m$ and assume that whenever $g_i|_{Z = 0} \neq 0$, we have $\init_<(g_i) = \init_<(g_i|_{Z = 0})$. Then the non-zero $g_i|_{Z = 0}$ are a Gröbner basis for $J|_{Z = 0}$, both as an ideal of $K[x_1, \ldots, x_m]$ and of $K[\{x_1, \ldots, x_m\} \setminus Z]$, each time with respect to $<$. In other words, $\init_<(J|_{Z=0}) = \init_<(J)|_{Z=0}$, both as ideals of $K[x_1, \ldots, x_m]$ and of $K[\{x_1, \ldots, x_m\} \setminus Z]$.
\end{lemma}

\begin{proof}
    We first consider $J|_{Z = 0}$ as an ideal of the smaller polynomial ring $K[\{x_1, \ldots, x_m\} \setminus Z]$ and show that the $g_i|_{Z=0}$ form a Gröbner basis with respect to $<$. For this, let $f \in J$. We need to prove that if $f|_{Z=0} \neq 0$, then $\init(f|_{Z = 0})$ is divisible by some $\init(g_i|_{Z=0})$. For this, we write $f = p + q$ where $p$ is the sum of all terms of $f$ not divisible by any variable in $Z$ and $q = f - p$, i.e., $q$ is the sum of all terms of $f$ divisible by at least one variable in $Z$. Clearly, $p|_{Z = 0} = p$ and $q|_{Z = 0} = 0$, so $f|_{Z = 0} = p$, and we assume $p \neq 0$. If the initial term of $f$ is a summand of $p$, then we are done since the $g_i$ form a Gröbner basis for $J$. Otherwise, the leading term $m$ of $f$ is from $q$. Then there is some $g_i$ and a scalar multiple of a monomial, say $n$, such that $m = n \cdot \init(g_i)$. Since $m|_{Z = 0} = 0$, at least one of $n$ and $\init(g_i)$ is divisible by some variable in $Z$. Consider now $f' \coloneqq f - n \cdot g_i \in J$ and write $f' = p' + q'$ as before. Then $p' = p$ and $\init(q') < \init(q)$. If now $\init(f')$ comes from $p' = p$, we are done. Otherwise the initial term of $f'$ lies in $q'$ again, and we continue in the same way. But we cannot always choose the initial term from the $q$-part since this would result in an infinite chain of strictly decreasing monomials with respect to the monomial order $<$, proving our claim.

    Secondly, the $g_i|_{Z=0}$ even form a Gröbner basis of $J|_{Z = 0}$ if the latter is considered as an ideal in the larger polynomial ring $K[x_1, \ldots, x_m]$. Indeed, write $J_1$ for the ideal in the smaller polynomial ring and $J_2$ for the ideal in the larger polynomial ring. Then
    \begin{equation*}
        J_2 \cong J_1 \otimes_{K[\{x_1, \ldots, x_m\} \setminus Z]} K[x_1, \ldots, x_m].
    \end{equation*}
    Let now $f \in J_2 \setminus \{0\}$. Then by the preceeding isomorphism, $f$ can be written uniquely as a sum of monomials in $Z$ with coefficients in $J_1$. The initial term of $f$ with respect to $<$ is therefore the initial term of some non-zero element of $J_1$ multiplied by some monomial in the $Z$-variables. This is a multiple of the initial term of an element of $J_1$ and therefore divisible by some $\init(g_i|_{Z=0})$ by the first part of the proof.
\end{proof}

\begin{cor}\label{cor:height}
    The $(n-1)$-minors of $X_G$ form a Gröbner basis of $I_{n-1}(X_G)$ with respect to $<_{T,G}$ with square-free initial ideal $I_T|_{Z = 0} \subseteq R$. In particular, $I_{n-1}(X_G)$ is always radical and
    \begin{equation*}
        \height(I_{n-1}(X_G)) = \begin{cases} 3 & \text{if } G \text{ is connected}, \\ 2 & \text{otherwise}. \end{cases}
    \end{equation*}
\end{cor}

\begin{proof}
    It only remains to prove the last claim. Since the polynomial ring $R$ is regular, $\dim(R/J) = \dim(R) - \height(J)$ for every ideal $J \subseteq R$. This proves $\height(I_{n-1}(X_G)) = \height(I_T|_{Z = 0})$. If $G$ is connected, $T$ is a spanning tree with vertex set $[n]$, and so $I_T|_{Z = 0} = I_T$. Let $\mathfrak{p}$ be a minimal prime ideal above $I_T$. As $\mathfrak{p}$ contains the first $n$ generators of $I_T$ which are the $n$ possible products of $n-1$ distinct diagonal variables, $\mathfrak{p}$ contains at least two distinct diagonal variables, say $x_{ii}$ and $x_{jj}$. All generators of $I_T$ which are not divisible by either of $x_{ii}$ and $x_{jj}$ are divisible by $x_p$, where $p$ is the unique path in $T$ between $i$ and $j$. Hence, every minimal prime $\mathfrak{p}$ above $I_T$ is generated by at least three distinct variables, and some are indeed generated by \emph{exactly} three distinct variables, so $\height(I_T) = 3$. If $G$ is not connected, let $i$ and $j$ be two vertices lying in different connected components. Then there is no path of $T$ involving both $i$ and $j$, hence $\mathfrak{p} = (x_{ii}, x_{jj})$ is a minimal prime above $I_T|_{Z = 0}$, so $\height(I_T|_{Z = 0}) = 2$ in this case.
\end{proof}

\section{Minimal free resolution of $I_{n-1}(X_G)$}\label{section_minimal_free_res}

\subsection{The minimal free resolution in the generic case}

We first recall a foundational result by J\'ozefiak. For this result, $S$ denotes any noetherian commutative ring with $1$ and $W$ any symmetric $n \times n$ matrix with coefficients in $S$. As usual, we denote by $S^{n \times n}$ the free $S$-module of $n \times n$ matrices and by $\tr \colon S^{n \times n} \rightarrow S$ the trace map whose kernel is a free $S$-module of rank $n^2 - 1$. By $\Sym^2(S^n)$ we denote the free $S$-module of symmetric $n \times n$ matrices and by $A_n(S)$ that of alternating matrices, i.e., the skew-symmetric $n \times n$ matrices with zeros on the diagonal.

\begin{thm}[{\cite[Theorem~3.1]{Jozefiak1978Ideals}}]
    Let $W \in \Sym^2(S^n)$ and $Y \in \Sym^2(S^n)$ the cofactor matrix of $W$. If the grade of $I_{n-1}(W) \subseteq S$ equals $3$, then the complex of free $S$-modules
    \begin{equation*}
        L(W) \colon \quad 0 \rightarrow A_n(S) \overset{d_3}{\longrightarrow} \ker(\tr \colon S^{n \times n} \rightarrow S) \overset{d_2}{\longrightarrow} S^{n \times n}/A_n(S) \overset{d_1}{\longrightarrow} S \rightarrow S/I_{n-1}(W) \rightarrow 0
    \end{equation*}
    is exact and provides a free resolution of $S/I_{n-1}(W)$. Here,
    \begin{align*}
        d_1(M \ \mathrm{mod} \ A_n(S)) &\coloneqq \tr (YM), \\
        d_2(N) &\coloneqq WN \ \mathrm{mod} \ A_n(S), \\
        d_3(A) &\coloneqq AW.
    \end{align*}
\end{thm}

We observe $S^{n \times n}/A_n(S) \cong S^{\binom{n+1}{2}} \cong \Sym^2(S^n)$, where we interpret $S^{\binom{n+1}{2}}$ as the $S$-module of upper triangular matrices. The first isomorphism is then given by taking any representing matrix $M \in S^{n \times n}$ and subtracting the alternating matrix from $M$ whose strict lower triangle (all entries strictly below the diagonal) agrees with that of $M$. The second isomorphism is given by taking the upper triangular matrix in $S^{\binom{n+1}{2}}$ and replacing its strict lower triangle of zeros by its flipped strict upper triangle.


Let us return to the case where $S = R$ is the polynomial ring. If all entries of $W$ are homogeneous of the same positive degree or zero and $\mathrm{grade}(I_{n-1}(W)) = 3$ (which is the largest possible value), then J\'ozefiak's result shows that $L(W)$ is the minimal graded free resolution of $R/I_{n-1}(W)$. If all non-zero entries of $W$ are of degree $1$, then all non-zero entries of $Y$ are of degree $n-1$, and hence $L(W)$ has the shape
\begin{equation*}
    0 \rightarrow R(-(n+1))^{\binom{n}{2}} \rightarrow R(-n)^{n^2 - 1} \rightarrow R(-(n-1))^{\binom{n+1}{2}} \rightarrow R \rightarrow R/I_{n-1}(X) \rightarrow 0.
\end{equation*}
In particular, this is true if $W$ is the sparse generic symmetric matrix $X_G$ where $G$ is \emph{connected on $[n]$} by Corollary~\ref{cor:height}, using that grade and height agree for ideals in a regular ring.

\subsection{Matrices representing the $d_i$}

We choose bases of the free $R$-modules in J\'ozefiak's minimal graded free resolution of $R/I_{n-1}(X)$. For $R^{n \times n}/A_n(R) = R(-(n-1))^{\binom{n+1}{2}}$ we choose the graded basis $E_{ij}$ for $i \leq j$, where $E_{ij}$ is the matrix with a $1$ in position $(i,j)$ and zeros everywhere else. We order this basis as
\begin{equation*}
    E_{11},E_{22}, \ldots, E_{nn}, E_{12}, \ldots, E_{1,n}, \ldots, E_{n-1,n}.
\end{equation*}
For $\ker(\tr: R^{n \times n} \rightarrow R) \cong R(-n)^{n^2 - 1}$ we take the graded basis
\begin{equation*}
    E_{22}-E_{11}, E_{33}-E_{11}, \ldots, E_{nn}-E_{11}, E_{12}, E_{21}, \ldots, E_{1,n},E_{n,1}, \ldots, E_{n-1,n},E_{n,n-1}.
\end{equation*}
Finally, for $A_n(R) \cong R(-(n+1))^{\binom{n}{2}}$ we take the graded basis
\begin{equation*}
    E_{12} - E_{21}, \ldots, E_{1,n} - E_{n,1}, \ldots, E_{n-1,n} - E_{n,n-1}.
\end{equation*}

The matrix of $d_1$, with respect to these bases, is simply the row vector whose entries are the cofactors of $X$, ordered in the same way as the basis of $R^{n \times n}/A_n(R)$, i.e., the principal minors come first and then all the others with the appropriate signs in the usual lexicographic order.

For $d_2$ we have
\begin{equation*}
    d_2(E_{ii}-E_{11}) = \sum_{\substack{k = 1 \\ k \neq i}}^n (-x_{1k}) E_{1k} + \sum_{k=2}^i x_{ki} E_{ki} + \sum_{k = i+1}^n x_{ik} E_{ik}
\end{equation*}
for $i > 1$ and
\begin{equation*}
    d_2(E_{ij}) = \sum_{k=1}^j x_{ki} E_{kj} + \sum_{k = j+1}^n x_{ki} E_{jk}.
\end{equation*}
for $i \neq j$. The matrix of $d_2$ can hence be written as a block matrix
\begin{equation*}
    [d_2] = \begin{pmatrix}
  \begin{matrix}
  -x_{11} & -x_{11} & -x_{11} & \hdots & -x_{11} \\
  x_{22} & 0 & 0 & \hdots & 0 \\
  0 & x_{33} & 0 & \hdots & 0 \\
  0 & 0 & x_{44} & \hdots & 0 \\
  \vdots & \vdots & \vdots & \ddots & \vdots \\
  0 & 0 & 0 & \hdots & x_{nn}
  \end{matrix}
  & \rvline &  \Gamma_2\\
\hline
  \ast & \rvline & \Delta_2
\end{pmatrix},
\end{equation*}
where the only non-zero entries of $\Gamma_2$ in the row corresponding to $E_{ii}$ is at the columns corresponding to $E_{ji}$ for some $j$, and the entry there is precisely $x_{ij}$ (the order of indices matters for the basis elements but not for the variables!). The only non-zero entries of $\Delta_2$ in the row corresponding to $E_{ij}$, $i < j$, is in the columns corresponding to $E_{ij}$, $E_{ji}$, $E_{ki}$ for $k \neq i,j$, and $E_{kj}$ for $k \neq i,j$, with entries $x_{ii}$, $x_{jj}$, $x_{jk}$, and $x_{ik}$, respectively. In particular, all non-zero entries of $[d_2]$ are simply variables up to sign. Moreover, all non-zero entries of $\Gamma_2$ are \emph{off-diagonal} variables, and no variable ever appears twice in the same row of $\Gamma_2$ or $\Delta_2$. This implies that, by setting to zero any set of entries of $\Gamma_2$, $\Delta_2$ and of the lower left block $\ast$, those columns of $[d_2]$ that do not become identically zero will remain linearly independent over the ground field $K$.

For $d_3$ we have
\begin{equation*}
    d_3(E_{ij} - E_{ji}) = x_{ij} (E_{ii} - E_{11}) - x_{ij} (E_{jj} - E_{11}) + x_{jj} E_{ij} - x_{ii} E_{ji} + \sum_{\substack{k = 1 \\ k \neq i,j}}^n x_{jk} E_{ik} + \sum_{\substack{k = 1 \\ k \neq i,j}}^n (-x_{ik}) E_{jk} 
\end{equation*}
for $i < j$. Again, all non-zero entries of $[d_3]$ are variables up to sign. If we write $[d_3]$ as a block matrix
\begin{equation*}
    [d_3] = \begin{pmatrix} \Gamma_3 \\ \hline \Delta_3 \end{pmatrix},
\end{equation*}
where the rows of $\Gamma_3$ correspond to the basis elements $E_{ii} - E_{11}$ for $i = 2, \ldots, n$ while the rows of $\Delta_3$ correspond to the remaining basis elements of $\ker(\tr)$, then no variable ever occurs twice in the same row of $[d_3]$. Moreover, all variables appearing in $\Gamma_3$ are off-diagonal and precisely one diagonal variable appears in every row of $\Delta_3$.

\subsection{Homogenizing $L(X)$}

Let $G$ be any undirected, simple graph on $[n]$. We will now homogenize $L(X)$ with respect to the weight $w_G$ from Section~\ref{section_groebner} by introducing a new variable $t$ with weight $1$. We start by homgenizing all $(n-1)$-cofactors which are the entries of the row vector $[d_1]$. We fix one cofactor. All terms of the latter with $w_G$-degree strictly less than the $w_G$-degree of the cofactor itself are now multiplied by the appropriate power of $t$, as usual. The resulting row vector is denoted $[d_1]^h$. Next, we fix a column of $[d_2]$. Let $d$ be the maximum $w_G$-degree of the product of an entry of this column with the corresponding $(n-1)$-cofactor of $[d_1]^h$. Then we multiply every entry of this column of $[d_2]$ by the appropriate power of $t$ such that the product of the resulting entry and the corresponding minor has combined degree equal to $d$. This we do for every column of $[d_2]$ and call the result $[d_2]^h$. Note that no column of $[d_2]^h$ is divisible by $t$. Similarly, we fix a column of $[d_3]$ and we choose any row of $[d_2]^h$. Then we homogenize this column of $[d_3]$ in the same way as we did for $[d_2]$. The result does not depend on the chosen row of $[d_2]^h$. Again, no column of $[d_3]^h$ will be divisible by $t$. By construction, the resulting sequence of matrices still gives a \emph{complex} $L(X)^h$, i.e. $[d_2]^h [d_3]^h = 0$ and $[d_1]^h [d_2]^h = 0$.

\begin{prop}\label{prop:res_initial}
    The complex $L(X)^h|_{t=0}$ is exact and therefore the graded minimal free resolution of $R/\init_{w_G}(I_{n-1}(X))$.
\end{prop}

\begin{proof}
    The reasoning is very similar to the proof of \cite[Proposition~3.8]{Boocher2012Paper}. Clearly, $L(X)^h|_{t=0}$ is a complex of free $R$-modules. The image of $[d_1]^h|_{t=0}$ agrees with $\init_{w_G}(I_{n-1}(X)) \subseteq R$ by Corollary~\ref{cor:generic_minors_groebner} since $<_{T,G}$ refines the weight order $<_{w_G}$. But we can say more. Recall that $I_T = \init_{<_{T,G}}(I_{n-1}(X)) = \init_{<_{T,G}} \init_{<_{w_G}}(I_{n-1}(X))$. The main result of \cite{Conca2020Squarefree} implies that the depth is preserved under square-free Gröbner degenerations. Hence, we first deduce that $R/I_T$ is Cohen--Macaulay becauso so is $R/I_{n-1}(X)$, and secondly we obtain that $R/\init_{<_{w_G}}(I_{n-1}(X))$ is Cohen--Macaulay since so is $R/I_T$. Moreover, the Hilbert functions of $I_T$, of $\init_{<_{w_G}}(I_{n-1}(X))$ and of $I_{n-1}(X)$ agree, so that all three ideals are Cohen--Macaulay of minimal multiplicity, and $I_T$ is an initial ideal of both $\init_{<_{w_G}}(I_{n-1}(X))$ and $I_{n-1}(X)$. From \cite[Example~1.2]{Conca2006Nice} it now follows that in fact all graded Betti numbers of these three ideals coincide.

    Therefore, it is now enough to show that, after tensoring with $R/\mathfrak{m} = K$, the dimension of the images of $[d_2]^h|_{t=0}$ and $[d_3]^h|_{t=0}$ agree with $\beta_{2,n}(R/I_{n-1}(X))$ and $\beta_{3,n+1}(R/I_{n-1}(X))$, respectively. This follows from the discussion above since no column of $[d_2]^h|_{t=0}$ and $[d_3]^h|_{t=0}$ is identically zero and no row of $\Gamma_2$, $\Delta_2$ or $[d_3]$ contains any variable more than once.
\end{proof}

\begin{thm}\label{thm:pruning_works}
    Let $G$ be any undirected, simple graph on $[n]$. The minimal graded free resolution of $R/I_{n-1}(X_G)$ is obtained from $L(X)$ via Boocher's pruning procedure.
\end{thm}

\begin{proof}
    The argument is exactly the same as that of Boocher in the proof of \cite[Theorem~4.1]{Boocher2012Paper} and follows from a careful study of the $w_G$-grading of the complex $L(X)^h$ together with Proposition~\ref{prop:res_initial}. We repeat the argument for the sake of completeness.
    We first observe that all principal minors have $w_G$-weight $2(n-1)$ and the same is true for $\det(X_{[n] \setminus k, [n] \setminus l})$, $k < l$, whenever there is a path in $G$ between $k$ and $l$. All the other $(n-1)$-minors of $X$ have weight $2n - 3$. The minors of the smaller $w_G$-weight $2n-3$ are precisely those which vanish identically after substituting zero for all variables in $Z$. All non-zero entries of the matrices $[d_2]$ and $[d_3]$ are simply variables up to sign and hence have $w_G$-weight at most $2$. Therefore, the minimal occurring $w_G$-degree in the $w_G$-graded free $R[t]$-module of $L(X)^h$ at the $i$-th place is $-2(n-1)-2(i-1)$ for all $i \geq 1$. Therefore, with respect to the $w_G$-grading, $L(X)^h$ has the following shape:
    \begin{align*}
        {\mbox{$\begin{array}{c} \bigoplus_{c_j<2n+2} R[t](-c_j) \\ \bigoplus \\ \bigoplus R[t](-2n-2) \end{array}$}} \overset{[d_3]^h}{\longrightarrow} {\mbox{$\begin{array}{c} \bigoplus_{b_j<2n} R[t](-b_j) \\ \bigoplus \\ \bigoplus R[t](-2n) \end{array}$}} \overset{[d_2]^h}{\longrightarrow} {\mbox{$\begin{array}{c} \bigoplus_{a_j<2(n-1)} R[t](-a_j) \\ \bigoplus \\ \bigoplus R[t](-2(n-1)) \end{array}$}} \overset{[d_1]^h}{\longrightarrow} R[t].
    \end{align*}
    We write
    \begin{equation*}
        [d_i]^h = \begin{pmatrix}
            A_i
  & \rvline &  B_i\\
\hline
  C_i & \rvline & D_i
        \end{pmatrix}
    \end{equation*}
    for $i = 2, 3$ and $[d_1]^h = \begin{pmatrix} C_1 & \rvline & D_1 \end{pmatrix}$, each time according to the direct sum decomposition. Just as Boocher, we can deduce three things from the grading alone:
    \begin{itemize}
        \item For $i = 2,3$, every non-zero entry of $B_i$ has $w_G$-weight at least $3$ and is therefore divisible by $t$. In particular, $B_i|_{t=0} = 0$.
        \item For $i = 2,3$, every non-zero entry of $D_i$ has $w_G$-weight $2$ and is therefore either of the form $\pm x_{ij}$ for $ij \in G$ or $i = j$ or of the form $\pm t x_{ij}$ for $i < j$ and $ij \not\in G$. In particular, $D_i|_{t=0} = D_i|_{Z=0}$. The latter is true also for $i = 1$.
        \item For $i = 2,3$, all non-zero entries of $C_i$ necessarily have $w_G$-weight $1$ and are hence of the form $\pm x_{ij}$ for $i<j$ and $ij \not\in G$. Therefore, $C_i|_{Z=0} = 0$ for all $i = 1,2,3$.
    \end{itemize}
    After setting $t=0$ in $L(X)^h$, we obtain the maps
    \begin{align*}
    [d_1]^h|_{t=0} &= \begin{pmatrix} C_1 & \rvline & D_1|_{t=0} \end{pmatrix}, \\
    [d_i]^h|_{t=0} &= \begin{pmatrix}
        A_i|_{t=0} & \rvline &  0\\
        \hline
        C_i & \rvline & D_i|_{t=0}
        \end{pmatrix} \quad \text{for } i=2,3,
    \end{align*}
    giving an exact sequence by Proposition~\ref{prop:res_initial}. Now, tensoring with $R/(Z)$ does not change $D_i|_{t=0}$ for $i = 2,3$ while $C_i$ becomes the zero matrix for all $i = 1,2,3$. The pruning procedure will hence first erase all columns of $C_1$, so that $[d_1]^h|_{t=0}$ becomes simply $D_1|_{t=0}$, whose image is $I_{n-1}(X_G)$. The corresponding rows of $[d_2]^h|_{t=0}$, which are precisely those of $A_2$, will be erased as well. Proceeding in the same way, the resulting pruned complex $P_G$ eventually only consists of the three maps $D_1, D_2, D_3$ with all variables from $Z$ replaced by zeros. The complex $P_G$ is then still exact: Denote by $\pi_i\colon (L(X)^h|_{t=0})^i \to P_G^i$ the projection onto the second big direct summand. Then $\pi_i([d_{i+1}]^h|_{t=0}(v)) - D_{i+1}|_{Z=0}(\pi_i(v))$ lies in the image of $C_{i+1}$ whose non-zero entries are variables in $Z$ up to sign. If, over $R/(Z)$, the element $w$ is in the kernel of $D_i$, we consider the obvious lift of $w$ to $(0,w)$ in the kernel of $[d_i]^h|_{t=0}$, now over $R$. By exactness of $L(X)^h|_{t=0}$, there exists $v$ such that $(0,w) = [d_{i+1}]^h|_{t=0}(v)$ and hence
    \begin{equation*}
        w = \pi_i((0,w)) = \pi_i([d_{i+1}]^h|_{t=0}(v)) = D_{i+1}|_{Z=0}(\pi_i(v)) \ \mathrm{mod} \ Z,
    \end{equation*}
    so that over $R/(Z)$ the element $w$ lies in the image of $D_{i+1}|_{Z = 0}$. This shows that $P_G$ is exact over $R/(Z)$ and hence also over $R$ because $R$ is a free module over $R/(Z)$.
\end{proof}

\begin{cor}
    The graded Betti numbers of $R/I_{n-1}(X_G)$ are those stated in Theorem~\ref{thm:main}.
\end{cor}

\begin{proof}
    By Theorem~\ref{thm:pruning_works}, it is enough to understand which columns of the matrices $[d_i]|_{Z=0}$ are identically zero. First, the entries of the row vector $[d_1]|_{Z=0}$ are $\pm \det((X_G)_{[n] \setminus k, [n] \setminus l})$ for $k \leq l$ and the latter vanishes if and only if $k \neq l$ and there is no path between $k$ and $l$ in $G$. Next, we consider our explicit description of the matrix $[d_2]$. We claim that, after pruning, the column of $[d_2]|_{Z=0}$ corresponding to the basis element $E_{ij}$, $i \neq j$, is zero precisely if there is no path in $G$ between $i$ and $j$. The entry in the row corresponding to $E_{ij}$, if $i < j$, resp. to $E_{ji}$, if $i > j$, is the diagonal variable $x_{ii}$. Thus, for the column to vanish this row must have been erased in the pruning process, which is the case if and only if the minor corresponding to $ij$ vanishes identically after substituting zero for all variables in $Z$. This, as we saw, is equivalent to $i$ and $j$ lying in different connected components of $G$. Conversely, if this is the case, then the column of $[d_2]|_{Z=0}$ indeed vanishes after pruning because the remaining variables appearing in this column are $x_{ki}$ for $k \neq i$ which appears only in the row corresponding to $E_{kj}$ if $k \leq j$ resp. $E_{jk}$ if $k>j$. But either $x_{ki} \in Z$ or $ki \in G$, and in the last case $k$ and $j$ necessarily lie in different connected components of $G$, so the row corresponding to $E_{kj}$ resp. $E_{jk}$ must have been erased in the pruning process. The argument for $[d_3]$ is similar.
\end{proof}

\section{A first non-trivial characteristic number for smooth sparse quadrics}

Let $K = \overline{K}$ and $\chr(K) \neq 2$ in this section. The set of quadric hypersurfaces in $\PP^{n-1}_K$ is identified with the set of non-zero symmetric $n \times n$ matrices over $K$ up to scaling, i.e., with $\PP(\Sym^2(K^n))$. A $G$-sparse quadric is one where all off-diagonal entries of its associated symmetric matrix corresponding to the non-edges of $G$ are zero. Geometrically, this is a coordinate linear subspace of $\PP(\Sym^2(K^n))$, hence is itself a projective space and clearly agrees with $\Proj(R_G) = \PP^{N_G - 1}$, where $R_G \coloneqq R/(Z)$ and $N_G \coloneqq \dim(R_G) = |E_G| + n$, $E_G$ being the edge set of $G$.

An immediate consequence of Theorem~\ref{thm:main} is that we can compute the degree of the vanishing subscheme $V(I_{n-1}(X_G)) \subseteq \PP^{N_G - 1}$.

\begin{prop}\label{prop:degree}
    If $G$ is a disconnected graph on $[n]$, the codimension of $V(I_{n-1}(X_G)) \subseteq \PP^{N_G - 1}$ is $2$ and its degree is precisely $D_G$. If $G$ is connected on $[n]$, then the codimension of $V(I_{n-1}(X_G)) \subseteq \PP^{N_G - 1}$ is $3$ and its degree is precisely $\binom{n+1}{3}$.
\end{prop}

\begin{proof}
    By Theorem~\ref{thm:main}, the minimal graded free resolution of $R_G/I_{n-1}(X_G)$ over $R_G$ (which looks the same as over $R$) has the form
    \begin{equation*}
        0 \rightarrow R_G(-(n+1))^{\binom{n}{2} - D_G} \rightarrow R_G(-n)^{n^2 - 1 - 2D_G} \rightarrow R_G(-(n-1))^{\binom{n+1}{2} - D_G} \rightarrow R_G.
    \end{equation*}
    For the Hilbert series of $R_G/I_{n-1}(X_G)$ we deduce
    \begin{align*}
        \mathrm{HS}(R_G/I_{n-1}(X_G)) &= \sum_{d = 0}^\infty \HF(R_G/I_{n-1}(X_G))(d) t^d \\
        &= \sum_{d = 0}^\infty \binom{d+N_G-1}{N_G-1} t^d \\
        &\quad - \left( \binom{n+1}{2} - D_G \right) \sum_{d = 0}^\infty \binom{d-(n-1)+N_G-1}{N_G-1} t^d \\
        &\quad + (n^2 - 1 - 2D_G) \sum_{d = 0}^\infty \binom{d-n+N_G-1}{N_G-1} t^d \\
        &\quad - \left( \binom{n}{2} - D_G\right) \sum_{d = 0}^\infty \binom{d-(n+1)+N_G-1}{N_G-1} t^d \\
        &= \frac{1 - \left( \binom{n+1}{2} - D_G \right)t^{n-1} + (n^2 - 1 - 2D_G)t^n - \left( \binom{n}{2} - D_G \right)t^{n+1}}{(1-t)^{N_G}} \\
        &= \frac{t^{n-1}(1-t)^2 D_G + (1-t)^3 \sum_{k=0}^{n-2} \binom{k+2}{2}t^k}{(1-t)^{N_G}}.
    \end{align*}
    Observing that $\sum_{k=0}^{n-2} \binom{k+2}{2} = \binom{n+1}{3}$, the last computation gives all claims after canceling $(1-t)^2$ for $D_G \neq 0$ or $(1-t)^3$ for $D_G = 0$.
\end{proof}

An immediate application of Proposition~\ref{prop:degree} is the following geometric result.

\begin{cor}
    Let $n \geq 3$. For any graph $G$ on $[n]$, the number of smooth $G$-sparse quadrics in $\PP^{n-1}$ tangent to $2$ general hyperplanes and passing through $N_G - 3$ general points is
    \begin{equation*}
        (n-1)^2 - D_G.
    \end{equation*}
    If $G$ is a connected graph on $[n]$, then moreover the number of smooth $G$-sparse quadrics in $\PP^{n-1}$ tangent to $3$ general hyperplanes and passing through $N_G - 4$ general points is
    \begin{equation*}
        (n-1)^3 - \binom{n+1}{3} = \frac{(n-1)(n-2)(5n-3)}{6}.
    \end{equation*}
\end{cor}

\section{Outlook}

\subsection{Primality} It would be desirable to have a combinatorial characterization for when $I_{n-1}(X_G)$ is prime, and this same question can of course be asked for the ideals of minors of arbitrary size $I_k(X_G)$, $1 \leq k \leq n$. It is clear that $I_1(X_G)$ is always a prime ideal but even for $k = 2$ and $k = n$ the answer is not entirely trivial. A necessary condition for $I_k(X_G)$ to be prime where $2 \leq k \leq n$ is that $G$ is $(n-k+1)$-connected, i.e., for any subset $M \subseteq [n]$ of cardinality $|M| = k$, the induced subgraph of $G$ on $M$ is connected. Indeed, if $G|_M$ is disconnected, then the principal minor $\det((X_G)_{M, M})$ factors as a product of two lower-order principal minors because after some permutation of $M$ the matrix $(X_G)_{M, M}$ is block-diagonal. But none of the two factors can be contained in $I_k(X_G)$ for degree reasons. This observation is also present in \cite[Lemma~7.10]{Conca2019Lovasz}. The following example is due to Aldo Conca. It shows that this necessary combinatorial condition is \emph{not} sufficient in general.

\begin{ex}\label{ex:Aldo}
    Let $n = 8$, $k = 5$ and $\chr(K) = 0$. Let $H$ be the complement graph of the complete bipartite graph $K_{4,4}$ with partition $[8] = \{1,2,3,4\} \sqcup \{5,6,7,8\}$. Adding to $H$ the four additional edges $15$, $26$, $37$ and $48$, we obtain a graph $G$. The induced subgraph of $G$ on any set of $5$ vertices is connected, so $G$ is $4$-connected. The sparse generic symmetric matrices for $G$ and $H$ are:
    \begin{align*}
        X_G &=
        {\small
        \begin{pmatrix}
            x_{11} & x_{12} & x_{13} & x_{14} & x_{15} & 0 & 0 & 0 \\
            x_{12} & x_{22} & x_{23} & x_{24} & 0 & x_{26} & 0 & 0 \\
            x_{13} & x_{23} & x_{33} & x_{34} & 0 & 0 & x_{37} & 0 \\
            x_{14} & x_{24} & x_{34} & x_{44} & 0 & 0 & 0 & x_{48} \\
            x_{15} & 0 & 0 & 0 & x_{55} & x_{56} & x_{57} & x_{58} \\
            0 & x_{26} & 0 & 0 & x_{56} & x_{66} & x_{67} & x_{68} \\
            0 & 0 & x_{37} & 0 & x_{57} & x_{67} & x_{77} & x_{78} \\
            0 & 0 & 0 & x_{48} & x_{58} & x_{68} & x_{78} & x_{88}
        \end{pmatrix},
        } \\
        X_H &= 
        {\small
        \begin{pmatrix}
            x_{11} & x_{12} & x_{13} & x_{14} & 0 & 0 & 0 & 0 \\
            x_{12} & x_{22} & x_{23} & x_{24} & 0 & 0 & 0 & 0 \\
            x_{13} & x_{23} & x_{33} & x_{34} & 0 & 0 & 0 & 0 \\
            x_{14} & x_{24} & x_{34} & x_{44} & 0 & 0 & 0 & 0 \\
            0 & 0 & 0 & 0 & x_{55} & x_{56} & x_{57} & x_{58} \\
            0 & 0 & 0 & 0 & x_{56} & x_{66} & x_{67} & x_{68} \\
            0 & 0 & 0 & 0 & x_{57} & x_{67} & x_{77} & x_{78} \\
            0 & 0 & 0 & 0 & x_{58} & x_{68} & x_{78} & x_{88}
        \end{pmatrix}.
        }
    \end{align*}
    We claim that $I_5(X_G)$ is not prime. For this, let $J = (x_{15}, x_{26}, x_{37}, x_{48})$. Then, clearly $I_5(X_G) + J = I_5(X_H) + J$. Since $X_H$ is block-diagonal, every $5$-minor of $X_H$ factors as a product of either a variable in one block and the determinant of the other block or as the product of a $2$-minor of one block and a $3$-minor of the other block. Laplace expansion shows that the determinant of a square matrix is contained in the ideal of all submaximal minors. In particular, $I_5(X_H)$ is contained in $I_1 + I_2$, where $I_1$, $I_2$ are the ideals of all $3$-minors of the two blocks, respectively. It is well-known that $I_1$ and $I_2$ are of height $3$. In particular, $I_1 + I_2 + J$ is a prime ideal of height $3+3+4 = 10$ since the three ideals are geometrically prime and involve disjoint sets of variables. We obtain $I_5(X_G) \subseteq I_5(X_G) + J = I_5(X_H) + J \subseteq I_1 + I_2 + J$. On the other hand, a computation in \texttt{Macaulay2} for $K = \Q$ gives $\height I_5(X_G) = 10$ over $\Q$ and hence over any field $K$ of $\chr(K) = 0$. So if $I_5(X_G)$ was prime, necessarily $I_5(X) = I_1 + I_2 + J$, which is impossible for degree reasons.
\end{ex}

Nonetheless, for some values of $k$ the necessary combinatorial condition is actually sufficient as the next result states.

\begin{prop}
    For $k = n$, the principal ideal $I_n(X_G) = (\det(X_G))$ is prime if and only if $G$ is connected. In case $\chr(K) = 0$ and $k = 2,3$, again $I_k(X_G)$ is prime if and only if $G$ is $(n-k+1)$-connected.
\end{prop}

\begin{proof}
    For $k = n$, it is enough to prove that $\det(X_G)$ is an irreducible polynomial in $R$. We adapt a combinatorial proof of the case $G = K_n$.\footnote{see the accepted answer of \href{https://math.stackexchange.com/questions/1893344/determinant-of-symmetric-matrix-is-an-irreducible-polynomial}{this} stackexchange post}
    First, we recall the following elementary fact: Let $S$ be any integral domain and $f \in S[x]$ a polynomial over $S$ in a single variable $x$. If $f = ax + b$ for $a,b \in S$, $a \neq 0$, and $f$ factors as $f = gh$ in $S[x]$, then precisely one of $g$ and $h$ is linear in $x$, i.e. of the form $cx+d$, $c \neq 0$, and the other one does not involve $x$ at all.
    
    Let now $f \coloneqq \det(X_G)$ where $G$ is connected. We have the following two facts:
    \begin{itemize}
        \item $f$ is linear in $x_{ii}$ for all $i = 1, \ldots, n$.
        \item For all $i < j$ such that $ij \in G$, we have that $f$ contains terms which are divisible by $x_{ij}$ but no term of $f$ is divisible by $x_{ii} x_{ij}$ or $x_{jj} x_{ij}$.
    \end{itemize}
    Now we assume that $f = gh$ in $R$. Without loss of generality, we assume that $g$ is linear in $x_{11}$, hence $h$ is independent of $x_{11}$ by the above. Therefore, $h$ is also independent of $x_{1i}$ for all $i > 1$ such that $1i \in G$ since otherwise $f$ would contain some term divisible by $x_{11} x_{1i}$. Hence, $g$ contains terms divisible by $x_{1i}$ for all $i > 1$ such that $1i \in G$ since $f$ contains such terms and $h$ does not. We conclude that all non-zero entries of $X_G$ in the first row appear only in $g$ but not in $h$. Next, for every $i$ such that $1$ is incident to $i$ in $G$, $h$ must be independent of $x_{ii}$ as well, otherwise $f$ would again contain terms of the form $x_{ii} x_{1i}$. Hence, $g$ must be linear in $x_{ii}$ for all $i$ incident to $1$ in $G$. With the same argument as before, $h$ is then also independent of all $x_{ij}$ such that $ij \in G$ while $g$ contains terms divisible by each of these variables. We conclude that all non-zero entries of $X_G$ in the $i$-th row only appear in $g$ but not in $h$ for every $i$ such that $1i$ is an edge of $G$. Continuing in this way, since $G$ is connected, we will eventually reach every row of $X_G$, implying that $h$ is constant, which concludes the proof.

    For $k = 2,3$ and $\chr(K) = 0$, the second claim is just a reformulation of \cite[Theorem~7.8]{Conca2019Lovasz} since their graph is precisely the complement graph $G^c$ of $G$. The important observation for $k = 2,3$ is that the complement of any $(n-k+1)$-connected graph $G$ is automatically a forest of maximal degree at most $k-2$. Indeed, for $k=2$ the graph $G$ is $(n-1)$-connected if and only if $G$ induces a connected graph on any pair of vertices. This means that $G = K_n$ is the complete graph. Hence, $G^c$ has no edges at all, so $G^c$ is clearly a forest of maximal degree $0$. For $k=3$ the graph $G$ is $(n-2)$-connected if and only if $G$ induces a connected graph on any triple of vertices, so either the triangle or the path. For $G^c$ this means that there is at most one edge between any three vertices. In particular, every vertex has degree at most $1$ in $G^c$, so $G^c$ is a forest of maximal degree $1$.
\end{proof}

Motivated by these results, we ask the following question.

\begin{question}
    Is the ideal $I_{n-1}(X_G)$ prime whenever $G$ is $2$-connected?
\end{question}

For $K = \Q$, we checked with \texttt{Macaulay2} that the answer is affirmative for $n \leq 6$.

\subsection{More sparsity} Graphical models in algebraic statistics motivate the study of ideals generated by only \emph{some} submaximal minors of a generic symmetric matrix. If the matrix in addition is allowed to be sparse, one arrives at the notion of a so-called \emph{double Markovian model} \cite{DoubleMarkovian}. The following result answers a combinatorial question raised in \cite[Remark~9]{DoubleMarkovian}, showing that the positivity assumption of \cite[Corollary~3]{DoubleMarkovian} can be relaxed.

\begin{prop}\label{prop:principally_regular}
	Let $A \in \Sym^2(K^n)$ be a symmetric $n \times n$ matrix which is principally regular, i.e., all principal minors of $A$ are non-zero. If for all $1 \leq i < j \leq n$ we have $A_{ij} \cdot (A^{-1})_{ij} = 0$, then $A$ is a diagonal matrix.
\end{prop}

\begin{proof}[Proof of Proposition~\ref{prop:principally_regular}]
	We write $A = (a_{ij})_{i,j \in [n]}$. First note that any principal submatrix of a principally regular symmetric matrix is also principally regular by definition. Moreover, an invertible symmetric matrix is principally regular if and only if so is its inverse. The latter follows from the formula
	\begin{equation*}
	\det((A^{-1})_{I,I}) = \frac{\det(A_{[n] \setminus I, [n] \setminus I})}{\det(A)}
	\end{equation*}
	for any $I \subseteq [n]$. If $A$ is not diagonal, then after a permutation of $[n]$ we may assume that there exists a non-zero off-diagonal entry in the first row. We can even assume that there is $k \geq 2$ such that $a_{1i} = 0$ for all $i \geq k+1$ and $a_{1i} \neq 0$ for all $2 \leq i \leq k$. Then
	\begin{equation*}
	A_{[n] \setminus 1, [n] \setminus i} = \begin{pmatrix}
	\begin{matrix}
	a_{12} \\ a_{13} \\ \vdots \\ a_{1k} \\ 0 \\ \vdots \\ 0
	\end{matrix} & \vline & A_{[n] \setminus 1, [n] \setminus 1i}
	\end{pmatrix}.
	\end{equation*}
	By hypothesis it now follows that $(A^{-1})_{1i} = 0$ for all $2 \leq i \leq k$ which translates into
	\begin{equation*}
	0 = (-1)^i \det(A_{[n] \setminus 1, [n] \setminus i}) = \sum_{j=2}^k (-1)^{i+j} a_{1j} \det(A_{[n] \setminus 1j, [n] \setminus 1i})
	\end{equation*}
	for all $2 \leq i \leq k$. Equivalently, in matrix form,
	\begin{equation}\label{eq:principally_regular}
	\begin{pmatrix}
	(-1)^{i+j} \det(A_{[n] \setminus 1j, [n] \setminus 1i})
	\end{pmatrix}_{i,j = 2, 3, \ldots, k} \cdot \begin{pmatrix}
	a_{12} \\ a_{13} \\ \vdots \\ a_{1k}
	\end{pmatrix} =
	\begin{pmatrix}
	0 \\ 0 \\ \vdots \\ 0
	\end{pmatrix}.
	\end{equation}
	We observe that
	\begin{equation*}
	(-1)^{i+j} \det(A_{[n] \setminus 1j, [n] \setminus 1i}) = \det(A_{[n] \setminus 1, [n] \setminus 1}) \cdot ((A_{[n] \setminus 1, [n] \setminus 1})^{-1})_{i-1,j-1}.
	\end{equation*}
	But $(A_{[n] \setminus 1, [n] \setminus 1})^{-1}$ is a principally regular symmetric matrix, hence the $(k-1) \times (k-1)$ matrix in \eqref{eq:principally_regular} is invertible, implying $a_{12} = a_{13} = \cdots = a_{1k} = 0$.
\end{proof}

\subsection*{Acknowledgments}
Many thanks go to Aldo Conca for providing Example~\ref{ex:Aldo} and pointing us in the right direction with respect to Proposition~\ref{prop:principally_regular}.
Moreover, \texttt{Macaulay2} has been of great help in gaining intuition by computing many examples.
The first named author was supported by Mitacs through the RISE Germany 2022 program of the German Academic Exchange Service (DAAD).
The last named author is supported by the Deutsche Forschungsgemeinschaft (DFG, German Research Foundation) -- 314838170, GRK~2297 MathCoRe.

\bibliographystyle{amsplain}
\bibliography{SparseSymDet.bib}
\end{document}